\documentclass[oneside,11pt,reqno]{amsart}
\usepackage{amssymb,amsmath,amsthm,bbm,enumerate,mdwlist,url,multirow}
\usepackage[pdftex]{graphicx}
\usepackage{hyperref}
\newtheorem{definition}{Definition}%Extra square-bracket argument achives that the numbering is the same as for definition (single uniform counter). 
\newtheorem{proposition}[definition]{Proposition}
\newtheorem{theorem}[definition]{Theorem}
\newtheorem{remark}[definition]{Remark}
\newtheorem{corollary}[definition]{Corollary}
\numberwithin{equation}{section}
\numberwithin{definition}{section}
\def\PP{\mathsf P}
\def\EE{\mathsf E}
\def\supp{\mathrm{supp}}

\def\ZZ{\mathbb{Z}}
\def\LL{\mathcal L}
\def\tl{R}
\begin{document}
\title{A note on the times of first passage for `nearly right-continuous' random walks}

\author{Matija Vidmar}
\address{Department of Statistics, University of Warwick, UK}
\email{m.vidmar@warwick.ac.uk}

\thanks{The support of the Slovene Human Resources Development and Scholarship Fund under contract number 11010-543/2011 is acknowledged. 
%The author would also like to express his gratitude to Saul Jacka and Aleksandar Mijatovi\'c, whose discussions with regard to this work were invaluable. Any responsibility for eventual shortfalls must of course lie with the author alone.
}

\begin{abstract}
A natural extension of a right-continuous integer-valued random walk is one which can jump to the right by one or two units. First passage times above a given fixed level then admit a tractable Laplace transform (probability generating function). 
%A possible application to queuing theory and branching processes is considered. 
Explicit expressions for the probabilities that the respective overshoots are either $0$ or $1$, according as the random walk crosses a given level for the first time either continuously or not, also obtain. An interesting non-obvious observation, which follows from the analysis, is that any such (non-degenerate) random walk will, eventually in $n\in \mathbb{N}\cup \{0\}$, always be more likely to pass over the level $n$ for the first time with overshoot zero, rather than one. Some applications are considered. 
\end{abstract}



\maketitle

\section{Introduction}\label{section:Introduction}
It is well-known that, within the class of integer-valued random walks, those which can jump to the right by only one unit, are singled out in terms of having a more tractable fluctuation theory \cite{brown,quine}  \cite[Section~4]{vylder} \cite[Section~7]{dickson} \cite[Section~9.3]{doney} \cite[\emph{passim}]{spitzer}. For their defining property, they are called `right-continuous' or also `skip-free to the right'. In particular, first passage times above a given level then admit (semi)explicit Laplace transforms, at every point in terms of a single parameter. This is also by analogy to the spectrally negative class of L\'evy processes  \cite[Chapter VII]{bertoin} \cite[Section~9]{doney} \cite[Chapter 8]{kyprianou} \cite[Section 9.46]{sato}. Indeed, if the right-continuous integer-valued random walk is embedded into continuous time as a compound Poisson process \cite{vidmar:fluctuation}, then together (modulo trivial cases) these two types exhaust the class of L\'evy processes having non-random overshoots \cite{vidmar}, a property by and large responsible for the fluctuation theory then being more explicit. 

It seems natural to ask, then, to what extent fluctuation theory remains (and, for that matter, does not remain) tractable when the demand of non-random overshoots is relaxed. In this paper only the simplest extension is considered, namely we allow the random walk to jump to the right by one or two units (making it `nearly right-continuous', but not quite). Apart from such theoretical considerations, these `nearly right-continuous' random walks also extend some applied queuing and branching models related to right-continuous
random walks, lending further relevance to their study. 

Now, the mandate of this paper is restricted to establishing the Laplace transforms of the times of first passage above a given level, for the type of processes just described. It emerges that the values of the latter are determined at each point by two parameters, them in turn being characterized precisely in terms of the deterministic characteristics of the process. 
%At least in this limited sense, then, we lose almost nothing in the way of tractability. 
Moreover, along the way, we obtain the probabilities of the random walk crossing a given level for the first time either continuously or not. A couple of colloraries to these findings are made explicit, and some applications are considered.

%In particular, this allows to conclude that any such (non-degenerate) process will, eventually (in $n\in \mathbb{N}\cup \{0\}$) be more likely to cross the level $n$ for the first time continuously rather than dis-continuously. 
% --- which are interesting in their own right; particularly so, since fluctuation theory is, as indicated, seldom explicit, unless one is dealing with the case of a L\'evy process for which all the overshoots are actually $0$ or constant (almost surely). 

%Apart from the theoretical motivation, as given above, these `nearly right-continuous random walks' also naturally extend some applied queuing and branching models related to right-continuous random walks. This lends further relevance to their study. 

As regards the presentation of the remainder of this paper, for (notational) convenience, and also to make the connection to the fluctuation theory of L\'evy processes straightforward, we shall mostly work in the continuous-time compound Poisson setting. With this proviso, the main result of the paper is stated in Theorem~\ref{theorem} of Section~\ref{section:setting}, which also introduces the setting and fixes notation (Corollary~\ref{remark:remark} gives the result for the discrete-time setting and follows immediately from its continuous counter-part). Section~\ref{section:proof} contains the (not too difficult, still non-trivial) proof. In Section~\ref{section:application} we briefly remark upon some applications. Section~\ref{section:conclusion} concludes. 

\section{Setting, notation and statement of result}\label{section:setting}
The law of an integer-valued random walk $(W_k)_{k\geq 0}$, starting at $W_0=0$, is characterized through its transition probabilities $p_n:=\PP(W_1=n)$ ($n\in\mathbb{Z}$). For simplicity, insist on $p_0=0$. It is straightforward to see such a random walk as a continuous-time L\'evy (indeed, compound Poisson) process, by taking the L\'evy measure $\lambda:=\sum_{n\in\mathbb{Z}}p_n\delta_n$. Of course, the choice of the lattice $\mathbb{Z}$ rather than $h\mathbb{Z}$ (for some $h>0$) and the normalization $\lambda(\mathbb{R})=1$ are arbitrary. They are also natural and inconsequential, representing merely, and respectively, scalings of space and time.\footnote{Note that in this setting making $p_0$ non-zero corresponds simply to the lengthening of the expected holding period between jumps by a factor of $1/(1-p_0)$. %and this is again a trivial detail in light of our mandate!
}

It is thus no restriction if in the sequel we consider given on a probability space $(\Omega,\mathcal{F},\PP)$ a $\ZZ$-valued c\`adl\`ag L\'evy (hence compound Poisson) process $X$. Its L\'evy measure, which we insist on being of total mass $1$, will be denoted $\lambda$. Note we have not immediately assumed $X$ to be `nearly right-continuous' -- indeed, some auxiliary results will neither depend on this assumption, nor does proving or stating them for an arbitrary $\ZZ$-valued compound Poisson process of unit intensity represent any further discomfort.  For L\'evy processes refer to \cite{bertoin,sato} and with regard to their fluctuation theory \cite{doney,kyprianou}. The first passage times are defined as $T_x:=\inf\{t\geq 0: X_t\geq x\}$, $x\in \mathbb{R}$. 

It is trivial that for $t\geq 0$, $X_t$ admits exponential moments of all nonnegative orders, whenever $\lambda$ does so. Then $\EE[e^{\beta X_t}]=e^{t\psi(\beta)}$ (for $\beta\geq 0$, $t\geq 0$), where $\psi$, the Laplace exponent, given by $\psi(\beta):=\int\lambda(dx)(e^{\beta x}-1)$ (for $\beta\geq 0$), is continuous (dominated convergence). We make the standing assumption that $\lambda$ charges $(0,\infty)$, % (in light of our mandate, this is of course no restriction), 
which makes $\psi$ strictly convex (as follows via differentiation under the integral sign and the continuity of $\psi$) with $\lim_{+\infty}\psi=+\infty$. Hence, letting $\Phi(0)$ be the largest zero of $\psi$, $\psi\vert_{[\Phi(0),\infty)}:[\Phi(0),\infty)\to [0,\infty)$ is an increasing continuous  bijection; we may define $\Phi:=(\psi\vert_{[\Phi(0),\infty)})^{-1}$. Finally, when $\lambda$ has support bounded from above, we introduce $\Theta(\beta):=\int\lambda(dn)(\beta^n-1)$ (for $\beta\in \mathbb{R}\backslash (-1,1)$). %Note that $\Theta$ is continuous, by dominated convergence. 

The result may now be stated.

\begin{theorem}\label{theorem}
Suppose $X$ is `nearly skip-free to the right', i.e. $\supp( \lambda\vert_{\mathcal{B}((0,\infty))})\subset \{1,2\}$. Assume furthermore that $2\in \supp(\lambda)$ (so we are excluding the skip-free version) and $\supp (\lambda)\not\subset 2\ZZ$ (which is again the skip-free version but on double the lattice). 
\begin{enumerate}[(a)]
\item For every $q\geq 0$, such that $\Phi(q)>0$, and then for every $n\in \mathbb{N}\cup \{0\}$: 
%, there exist unique $\lambda_+(q)$ and $\lambda_-(q)$ satisfying $ \Theta(1/\lambda_{\pm}(q))=q$, and such that for all $n\in\mathbb{N}\cup \{0\}$: 
\begin{equation*}
\EE[e^{-q T_n}\mathbbm{1}(T_n<\infty)]=\frac{1-\lambda_-(q)}{\lambda_+(q)-\lambda_-(q)}\lambda_+(q)^{n+1}-\frac{1-\lambda_+(q)}{\lambda_+(q)-\lambda_-(q)}\lambda_-(q)^{n+1},
\end{equation*} where $1/\lambda_-(q)$ is the unique zero of $\Theta-q$ on $(-\infty,-1)$ and $1/\lambda_+(q)=e^{\Phi(q)}$ is the unique zero of $\Theta-q$ on $(1,\infty)$. Further, the following inequalities hold: $-\lambda_+(q)<\lambda_-(q)<0<\lambda_+(q)<1$. \label{theorem:a}
\noindent If $\Phi(0)=0$, then for each $n\in\mathbb{N}\cup \{0\}$, $\PP(T_n<\infty)=1$.
%, i.e. the process $X$ drifts to $+\infty$. 
\item For each $n\geq 0$: $\PP(X(T_n)=n,T_n<\infty)=\frac{1}{\lambda_+(0)-\lambda_-(0)}(\lambda_+(0)^{n+1}-\lambda_-(0)^{n+1})$ and $\PP(X(T_n)=n+1,T_n<\infty)=\frac{-\lambda_+(0)\lambda_-(0)}{\lambda_+(0)-\lambda_-(0)}(\lambda_+(0) ^n-\lambda_-(0)^n)$. Here $1/\lambda_+(0)=e^{\Phi(0)}$ is the largest zero of $\Theta$ on $[1,\infty)$ (where it has at most one in addition to $1$), and $1/\lambda_-(0)$ is the unique zero of $\Theta$ on $(-\infty,-1)$. In addition: $-\lambda_+(0)<\lambda_-(0)<0<\lambda_+(0)\leq 1$. \label{theorem:b}
%\item\label{theorem:band1/2} (i) $\lambda_+:[0,\infty)\to (0,\lambda_+(0)]$ and $\lambda_-:[0,\infty)\to (-\lambda_+(0),0)$ are continuous. (ii) $\lambda_+$ is decreasing to, whilst $\lambda_-$ is limiting to, zero. (iii) $\lambda_->-\lambda_+$. (iv) $\Theta\circ (1/\lambda_\pm)=\mathrm{id}_{[0,\infty)}$.
\item We have: $$\lim_{n\to\infty}\frac{\PP(X(T_n)=n+1\vert T_n<\infty)}{\PP(X(T_n)=n\vert T_n<\infty)}=-\lambda_-(0)\in (0,1).$$\label{theorem:c}
\item\label{theorem:d} $X$ drifts to $+\infty$, oscillates or drifts to $-\infty$ \cite[pp. 255-256, Proposition~37.10 and Definition~37.11]{sato}, according as to whether $\psi'(0+)>0$, $\psi'(0+)=0$ or $\psi'(0+)<0$.
\end{enumerate}
\end{theorem}

\begin{remark}
\leavevmode
\begin{enumerate}
\item It follows from Theorem~\ref{theorem}\eqref{theorem:a}-\eqref{theorem:b} and the continuity of $\Theta$ (which fact can be seen via dominated convergence, from the very definition of $\Theta$), that $\Theta_{(-\infty,1/\lambda_-(0)]}:(-\infty,1/\lambda_-(0)]\to [0,+\infty)$  is a decreasing bijection, $\Theta_{(1/\lambda_-(0),-1]\cup (1,1/\lambda_+(0))}$ is strictly negative, $\Theta_{[1/\lambda_+(0),\infty)}:[1/\lambda_+(0),\infty)\to [0,\infty)$ is an increasing bijection, $1/\lambda_-=(\Theta_{(-\infty,1/\lambda_-(0)]})^{-1}$ and $1/ \lambda_+=(\Theta_{[1/\lambda_+(0),\infty)})^{-1}$. See Figure~\ref{figure:nsf} for an illustration. 
\item We also see from Theorem~\ref{theorem}\eqref{theorem:c} that eventually (in $n\in \mathbb{N}\cup \{0\}$)  we will always be more likely to cross the level $n$ for the first time continuously rather than dis-continuously, \emph{no matter what the data}. Remark that if $\lambda(\{2\})$ is, \emph{ceteris paribus}, allowed to increase, we obtain that this does not necessarily hold for all $n\in \mathbb{N}\cup \{0\}$ (since we will be increasingly likely to cross over the level $1$ on the first jump by jumping to the level $2$)! 
\end{enumerate}
\end{remark}

\begin{figure}
                \includegraphics[width=\textwidth]{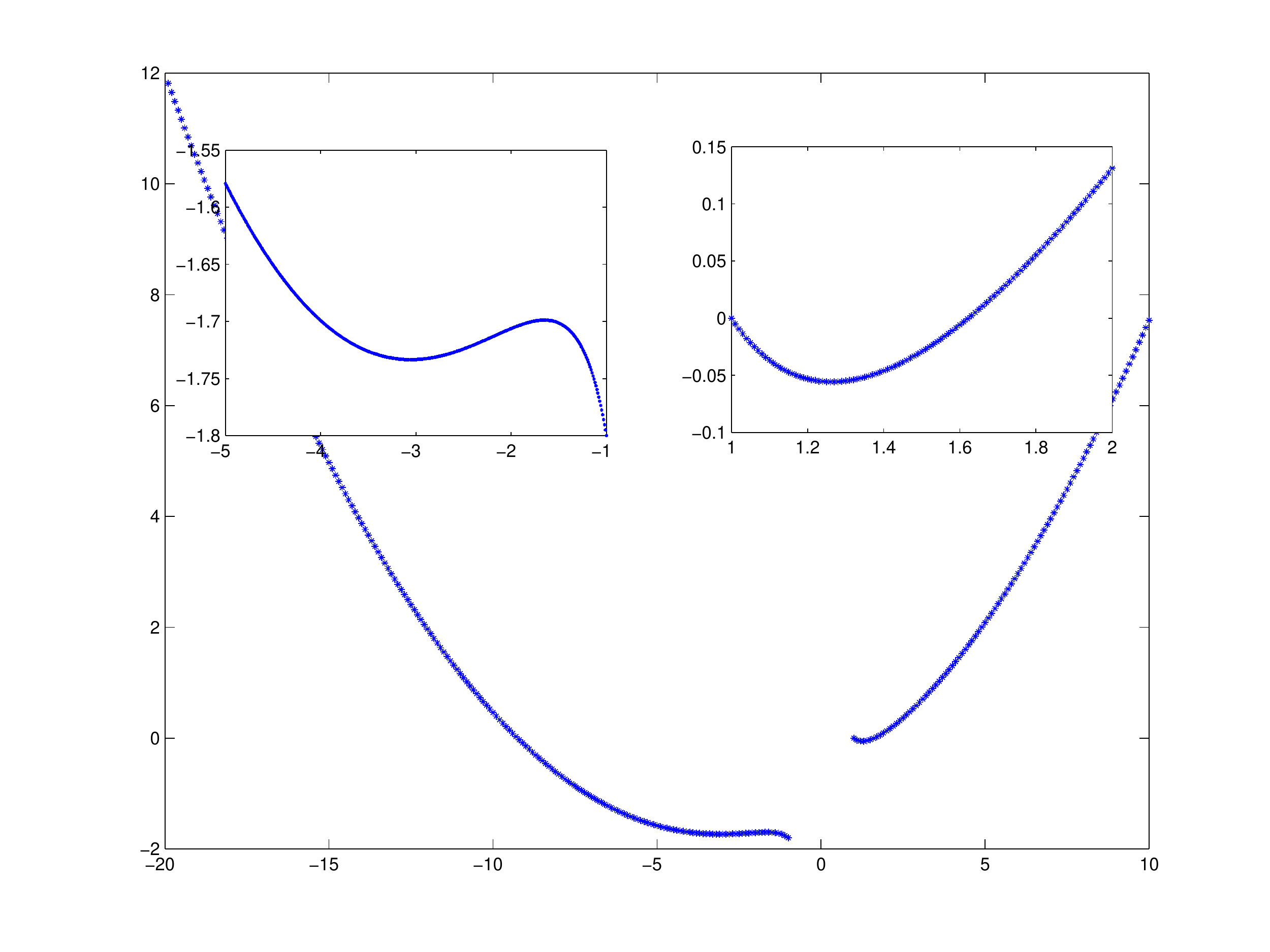}
\caption{The function $\Theta$ for the L\'evy measure $\lambda=0.05\delta_2+0.35\delta_1+0.4\delta_{-1}+0.05\delta_{-2}+0.15\delta_{-3}$.  On $[1,\infty)$, $\Theta$ is strictly convex. Its behavior on the interval $(-\infty,-1]$ is not trivial (left inset).}\label{figure:nsf}
\end{figure}

\begin{corollary}\label{remark:remark}
Returning to the original setting of a random walk $(W_k)_{k\geq 0}$, with possibly $p_0\ne 0$, and with $p_i=0$ for $i>2$, assume $p_2\ne 0$ and $p_i\ne 0$ for some odd integer $i$. 

Define $\LL(\beta):=\sum_{i\in \mathbb{Z}}p_i\beta^{-i}$ for $\beta\in [-1,1]\backslash \{0\}$ and note that $\LL(1)=1$; $\LL\vert_{(0,1]}$ is strictly convex; $\lim_{0+}\LL=+\infty$. Besides $1$ there is hence at most one other zero of $\LL-1$ on $(0,1]$, the smaller of the two is denoted %. We denote the smallest root of $(\LL-1)\vert_{(0,1]}$ by 
$\alpha(1)$. It is clear that $\LL\vert_{(0,\alpha(1)]}:(0,\alpha(1)]\to [1,\infty)$ is a decreasing bijection, and we let $\alpha$ be its inverse. 

Finally, introduce $\tilde{T}_n:=\inf\{k\geq 0:W_k\geq n\}$ ($n\in\mathbb{N}\cup \{0\}$). 
%\begin{enumerate}
%\item Returning to the original setting of a random walk $(W_k)_{k\geq 0}$, with possibly $p_0\ne 0$, and with $p_i=0$ for $i>2$, assume $p_2\ne 0$ and $p_i\ne 0$ for some odd integer $i$. 
%\item Define $\LL(\beta):=\sum_{i\in \mathbb{Z}}p_i\beta^{-i}$ for $\beta\in [-1,1]\backslash \{0\}$ and observe that $\LL(1)=1$ and that $\LL\vert_{(0,1]}$ is strictly convex. Moreover, $\lim_{0+}\LL=+\infty$. Other than $1$ there is hence at most one root of $\LL-1$ on $(0,1]$. We denote the smallest root of $(\LL-1)\vert_{(0,1]}$ by $\alpha(1)$. It is clear that $\LL\vert_{(0,\alpha(1)]}:(0,\alpha(1)]\to [1,\infty)$ is a decreasing bijection, and we let $\alpha$ be its inverse. 
%\item Introduce $\tilde{T}_n:=\inf\{k\geq 0:W_k\geq n\}$. 
\begin{enumerate}[(a)]
\item For any $\gamma\geq 1$ with $\alpha(\gamma)<1$, we have for all $n\in\mathbb{N}\cup \{0\}$: $$\EE[\gamma^{-\tilde{T}_n}\mathbbm{1}(\tilde{T}_n<\infty)]=\frac{1-\tilde{\lambda}_-(\gamma)}{\tilde{\lambda}_+(\gamma)-\tilde{\lambda}_-(\gamma)}\tilde{\lambda}_+(\gamma)^{n+1}-\frac{1-\tilde{\lambda}_+(\gamma)}{\tilde{\lambda}_+(\gamma)-\tilde{\lambda}_-(\gamma)}\tilde{\lambda}_-(\gamma)^{n+1}$$ where $\tilde{\lambda}_\pm(\gamma)$ is the unique zero of $\LL-\gamma$ on $\pm (0,1)$. Further, $\tilde{\lambda}_+(\gamma)=\alpha(\gamma)$ and the following inequalities hold: $-\tilde{\lambda}_+(\gamma)<\tilde{\lambda}_-(\gamma)<0<\tilde{\lambda}_+(\gamma)<1$. 

If $\alpha(1)=1$, then for each $n\in\mathbb{N}\cup \{0\}$, $\PP(\tilde{T}_n<\infty)=1$.
\item For each $n\geq 0$: $\PP(W(\tilde{T}_n)=n,\tilde{T}_n<\infty)=\frac{1}{\tilde{\lambda}_+(1)-\tilde{\lambda}_-(1)}(\tilde{\lambda}_+(1)^{n+1}-\tilde{\lambda}_-(1)^{n+1})$ and $\PP(W(\tilde{T}_n)=n+1,\tilde{T}_n<\infty)=\frac{-\tilde{\lambda}_+(1)\tilde{\lambda}_-(1)}{\tilde{\lambda}_+(1)-\tilde{\lambda}_-(1)}(\tilde{\lambda}_+(1) ^n-\tilde{\lambda}_-(1)^n)$. Here $\tilde{\lambda}_+(1)=\alpha(1)$ is the smallest zero of $\LL-1$ on $(0,1]$  and $\tilde{\lambda}_-(1)$ is the unique zero of $\LL-1$ on $(-1,0)$.  It holds: $-\tilde{\lambda}_+(1)<\tilde{\lambda}_-(1)<0<\tilde{\lambda}_+(1)\leq 1$. 
%\item (i) $\tilde{\lambda}_+:[1,\infty)\to (0,\tilde{\lambda}_+(1)]$ and $\tilde{\lambda}_-:[1,\infty)\to (-\tilde{\lambda}_-(1),0)$ are continuous. (ii) $\tilde{\lambda}_+$ is nonincreasing to, whilst $\tilde{\lambda}_-$ is limiting to, zero. (iii) $\tilde{\lambda}_->-\tilde{\lambda}_+$. (iv) $\LL\circ \tilde{\lambda}_\pm=\mathrm{id}_{[1,\infty)}$.
\item We have: $$\lim_{n\to\infty}\frac{\PP(W(\tilde{T}_n)=n+1\vert \tilde{T}_n<\infty)}{\PP(W(\tilde{T}_n)=n\vert \tilde{T}_n<\infty)}=-\tilde{\lambda}_-(1)\in (0,1).$$
\item $W$ drifts to $+\infty$, oscillates or drifts to $-\infty$, according as to whether $\LL'(1-)<0$, $\LL'(1-)=0$ or $\LL'(1-)>0$. 
\end{enumerate}
%\end{enumerate}
\end{corollary}

We conclude this section by further specifying the setting and fixing some relevant notation. Thus we allow $e_1$ to be an exponential, mean $1$, random variable, independent of $X$; then define $e_q:=e_1/q$ (for $q\in (0,\infty)\backslash \{1\}$) and agree $e_0=+\infty$, $\PP$-a.s. We also introduce the following quantities related to the laws of the overshoots (with $q\geq 0$): $\mu_x^q(A):=\EE[\mathbbm{1}(T_x<e_q)\mathbbm{1}_A\circ (X(T_x)-x)]$ (for $x\in\mathbb{R}$ and $A\in\mathcal{B}(\mathbb{R})$), whilst $p_n^i(q):=
%\PP(T_n<e_q,X(T_n)=n+i)=
\mu_n^q(\{i\})$ and $p_n(q):=
%\sum_{i=0}^\infty p_n^i(q)=
\mu_n^q(\mathbb{R})%=\PP(T_n<e_q)
$ (for $\{n,i\}\subset \mathbb{N}\cup \{0\}$). 

\section{Proof of theorem}\label{section:proof}
We return to the setting as described in Section~\ref{section:setting}, prior to stating Theorem~\ref{theorem}. The following simple proposition will give all the necessary 
%(and some which are not) 
ingredients towards the proof of Theorem~\ref{theorem}.

\begin{proposition}\label{prop:miscellaneous}
The following hold: 
\begin{enumerate}[(i)]
\item\label{prop:i} If $\supp (\lambda)$ is bounded from above, then: $\int \mu_x^q(dz)e^{\Phi(q)z}=e^{-\Phi(q)x}$ ($x\in\mathbb{R}$, $q\geq 0$). 
\item\label{prop:ii} $\mu_{x+y}^q(A)=\int \mu_x^q(du)\mu_{y-u}^q(A)$  ($\{x,y\}\subset [0,\infty)$, $q\geq 0$, $A\in\mathcal{B}(\mathbb{R})$).
\item\label{prop:iii} $\mu_x^q(A)=\mathbbm{1}_{(0,\infty)}(x)\frac{1}{1+q}\int \lambda(dv)\mu_{x-v}^q(A)+\mathbbm{1}_{(-A)\cap (-\infty,0]}(x)$ ($x\in\mathbb{R}$, $q\geq 0$, $A\in\mathcal{B}(\mathbb{R})$). 
\suspend{enumerate}
Moreover, if $X$ is `nearly right-continuous', i.e. $\supp (\lambda\vert_{\mathcal{B}((0,\infty))})\subset \{1,2\}$, then: 
\resume{enumerate}[{[(i)]}]
\item\label{prop:iv} For each $q\geq 0$, the sequences $(p_n^0(q))_{n\in\mathbb{N}\cup \{0\}}$ and $(p_n^1(q))_{n\in\mathbb{N}\cup \{0\}}$ satisfy the following system of linear difference equations:
\begin{eqnarray}
\label{recursion:one} p_{n+1}^0(q)&=&p_n^0(q)p_1^0(q)+p_n^1(q)\\
\label{recursion:two} p_{n+1}^1(q)&=&p_n^0(q)p_1^1(q)
\end{eqnarray}
($n\in\mathbb{N}\cup \{0\}$). 
\end{enumerate}
\end{proposition}
%
%\begin{remark}\label{remark}
%\begin{enumerate}
%\item It will be clear from the proof to what extent the above statements remain valid more generally than within the confines of our setting, simply by noting what are the respective reasons on account of which they hold true.
%\item\label{rem:2} In particular, if $\lambda$ on restriction to the positive half-line has bounded support contained in $\{1,\ldots,N+1\}$ ($N\in\mathbb{N}\cup \{0\}$), say, then in \eqref{prop:iv} we get a system of $N+1$ linear difference equations for the $N+1$ unknown sequences $(p_n^i)_{n\in\mathbb{N}\cup \{0\}}$, $i=0,\ldots, N$. In the `nearly right-continuous' case $N=1$, and for the right-continuous version $N=0$. 
%\end{enumerate}
%\end{remark}
\begin{proof}
Applying Optional Sampling Theorem to the bounded stopping time $T_x\land T$ ($T\geq 0$) and the martingale $(\exp\{\Phi(q)X_t-qt\})_{t\geq 0}$, 
%, which is so by the definition of the Laplace exponent and stationary independent increments of $X$. 
we obtain $\EE[\exp\{\Phi(q)X(T_x\land T)-q(T_x\land T)\}]=1$. Now let $T\to\infty$. Suppose first $\Phi(q)>0$. Then by dominated convergence, on account of the boundedness of $\supp (\lambda\vert_{\mathcal{B}((0,\infty))})$, $\EE[\exp\{\Phi(q)X(T_x)-qT_x\}\mathbbm{1}(T_x<\infty)]=1$, since a.s. on the event $\{T_x=\infty\}$, $X$ drifts to $-\infty$. Further, we can write the equality obtained, by using $e_q$, like so: 
\begin{equation}\label{prop:proof}
\EE[\exp\{\Phi(q)(X(T_x)-x)\}\mathbbm{1}(T_x<e_q)]=\exp\{-\Phi(q)x\}
%\EE[\mathbbm{1}(T_x<\infty)\exp\{\Phi(q)(X(T_x)-x)\}\int_0^\infty dy q e^{-qy}\mathbbm{1}_{(T_x,\infty)}(y)]=\exp\{-\Phi(q)x\}.
\end{equation}
%By the theorem of Fubini and properties of product measures (using independence of $e_q$ and $X$) this is the same as $\EE[\exp\{\Phi(q)(X(T_x)-x)\}\mathbbm{1}(T_x<e_q)]=\exp\{-\Phi(q)x\}$. 
Suppose now $\Phi(q)=0$. Then $q=0$. Letting $q\downarrow 0$ in \eqref{prop:proof}, using continuity of $\Phi$  and dominated convergence, \eqref{prop:proof} extends to this case as well. This concludes the proof of item \eqref{prop:i}.

\eqref{prop:ii} is the strong Markov property and the memoryless property of the exponential distribution, coupled with the independence of $X$ and $e_q$. 
%Also, since $X$ lives on $\ZZ$, measurability issues are straightforward. 

\eqref{prop:iii} follows for the same reasons, by conditioning on the first jump time of the process $X$ (the probability of seeing the latter before $e_q$ is $1/(1+q)$). 

Finally, to obtain \eqref{prop:iv}, condition on $T_n$, noting that in order to cross the level $n+1$, the random walk must first have crossed the level $n$, which it did either continuously or not. 
\end{proof}
Let us now apply the above to gain understanding of the `nearly right-continuous' random walk. We assume henceforth $\supp (\lambda\vert_{\mathcal{B}((0,\infty))})\subset \{1,2\}$. 
\begin{remark}
Suppose furthermore $\lambda(\{2\})=0$ for the right-continuous case. Then Proposition~\ref{prop:miscellaneous}\eqref{prop:i} yields at once $\PP(T_n<e_q)=e^{-\Phi(q)n}$, i.e. $\EE[e^{-q T_n}\mathbbm{1}(T_n<\infty)]=e^{-\Phi(q)n}$, for all $n\in\mathbb{N}\cup \{0\}$, $q\geq 0$.
\end{remark}
We now assume $\lambda(\{2\})>0$. If $\lambda$ is supported by $2\ZZ$, this is just the right-continuous case, but on the lattice $2\ZZ$. So without loss of generality take the converse case. In particular, it follows that $p_1^1(q)p_1^0(q)>0$ for all $q\geq 0$. 

In the first step towards the proof of Theorem~\ref{theorem}, we solve the recursion system of Proposition~\ref{prop:miscellaneous}\eqref{prop:iv}. Simply plug \eqref{recursion:two} into \eqref{recursion:one} to get (for any $q\geq 0$): $$p_{n+2}^0(q)-p_{n+1}^0(q)p_1^0(q)-p_n^0(q)p_1^1(q)=0,\quad n\in\mathbb{N}\cup \{0\}.$$ 
%Denote, provisionally, $a_1(q):=p_1^1(q)$ and $a_0(q):=p_1^0(q)$. 
The characteristic polynomial of this last recursion is (in the dummy variable $\zeta$) $\zeta^2-p_1^0(q)\zeta-p_1^1(q)$, with the zeros: $$\lambda_\pm(q):=\frac{p_1^0(q)}{2}\pm \sqrt{\left(\frac{p_1^0(q)}{2}\right)^2+p_1^1(q)}.$$ Note that $-\lambda_+(q)<\lambda_-(q)<0<\lambda_+(q)\leq 1$ (the last inequality follows from $p_1^0(q)+p_1^1(q)\leq 1$). %It is also clear, from the definitions, that $a_0(q)$ and $a_1(q)$ are nondecreasing as we let $q\downarrow 0$. Consequently $\lambda_+(q)$ is nondecreasing, as $q\downarrow 0$, and it converges to $\lambda_+(0)\in (0,1]$. In addition, $\lambda_-(q)$ converges to $\lambda_-(0)\in (-\lambda_+(0),0)$, as $q\downarrow 0$ (strict inequalities follow from $a_0(0)a_1(0)\ne 0$). 

Next, the initial values $p_0^i(q)=\delta_{0i}$, $i\in \{0,1\}$, yield immediately for all $n\in\mathbb{N}\cup \{0\}$, the following expressions: 
\begin{eqnarray}
\label{eq:3}p_n^0(q)&=&\frac{1}{\lambda_+(q)-\lambda_-(q)}(\lambda_+(q)^{n+1}-\lambda_-(q)^{n+1})\, \text{  and}\\
\label{eq:4}p_n^1(q)&=&\frac{-\lambda_+(q)\lambda_-(q)}{\lambda_+(q)-\lambda_-(q)}(\lambda_+(q) ^n-\lambda_-(q)^n),\text{ hence}\\
\label{eq:5}p_n(q)&=&\frac{-\lambda_+(q)\lambda_-(q)+\lambda_+(q)}{\lambda_+(q)-\lambda_-(q)}\lambda_+(q)^n-\frac{-\lambda_+(q)\lambda_-(q)+\lambda_-(q)}{\lambda_+(q)-\lambda_-(q)}\lambda_-(q)^n.
\end{eqnarray}

In the second step, we characterize the values of $\lambda_+(q)$ and $\lambda_-(q)$, $q\geq 0$. First, Proposition~\ref{prop:miscellaneous}\eqref{prop:i} implies $p_n^0(q)+p_n^1(q)e^{\Phi(q)}=e^{-\Phi(q) n}$ (for all $n\in\mathbb{N}\cup \{0\}$), from which it follows immediately that $\lambda_+(q)=e^{-\Phi(q) }$. If so, then the relation appearing in Proposition~\ref{prop:miscellaneous}\eqref{prop:i} is \emph{a priori} satisfied and does not yield $\lambda_-(q)$.  Indeed, as far as Theorem~\ref{theorem}\eqref{theorem:a} is concerned, if $q=0$ and $\Phi(q)=0$, so that $\lambda_+(q)=1$, we don't need it to, since then \eqref{eq:5} gives immediately that $\PP(T_n<\infty)=1$ for all $n\in \mathbb{N}\cup \{0\}$. Consider on the other hand the case when $\Phi(q)>0$ and hence $\lambda_+(q)<1$ (i.e. $1/\lambda_+(q)\in (1,\infty)$). In this instance we resort to  Proposition~\ref{prop:miscellaneous}\eqref{prop:iii} with $A=\mathbb{R}$, which tells us that (for all $n\in\mathbb{N}$): $$(1+q)p_n(q)=\sum_{k\in\ZZ}\lambda(\{k\}) p_{n-k}(q).$$ Plugging in \eqref{eq:5}, this implies (since $\Theta(1/\lambda_+(q))=\Theta(e^{\Phi(q)})=q$ and $-\lambda_+(q)\lambda_-(q)+\lambda_-(q)\ne 0$): $$\Theta(1/\lambda_-(q))=q,$$ where we know $1/\lambda_-(q)\in (-\infty,-1)$. 

Now, from the Introduction, it is clear that $\Theta-q$ has a unique zero on $(1,\infty)$, namely $e^{\Phi(q)}$. To establish 
Theorem~\ref{theorem}\eqref{theorem:a}, it will then be sufficient to argue that $\Theta-q$ has at most one zero on $(-\infty,-1)$ for each $q\geq 0$. Fix $q\geq 0$; let $\tl$ be any such zero.

It does not seem immediately clear analytically why $R$ should be unique (cf. Figure~\ref{figure:nsf}); so we argue probabilistically.\footnote{However, for numerical reasons, we note that $\lim_{-\infty}\Theta=+\infty$ and $\Theta(-1)<0$.} The argument is essentially verbatim that of the proof of Proposition~\ref{prop:miscellaneous}\eqref{prop:i}. First, $(\tl^{X_t}e^{-qt})_{t\geq 0}$ is a martingale. Indeed $\EE[\vert \tl^{X_t}\vert]=\EE[\vert \tl\vert^{X_t}]=e^{t\Theta(\vert \tl\vert)}<\infty$ and  $\EE[\tl^{X_t}]=e^{t\Theta(\tl)}=e^{t q}$. The assertion then follows by stationary independent increments of $X$. Further, for any $T\geq 0$ and $n\in\mathbb{N}\cup\{0\}$, Optional Sampling Theorem yields $\EE[\tl^{X(T_n\land T)}e^{-q(T_n\land T)}]=1$. Letting $T\to\infty$ we deduce by dominated convergence (as $X(T_n\land T)\leq n+1$ and since on the event $\{T_n=\infty\}$ the process $X$ limits to $-\infty$, a.s.): $\EE[\tl^{X(T_n)}e^{-q T_n}\mathbbm{1}(T_n<\infty)]=1$, i.e. $\EE[\tl^{(X(T_n)-n)}\mathbbm{1}(T_n<e_q)]=\tl^{-n}$. Put still another way, $p_n^0(q)+p_n^1(q)\tl=\tl^{-n}$. Since the left-hand side is a linear combination of $(n\mapsto \lambda_+(q)^n)$ and $(n\mapsto \lambda_-(q)^n)$ it follows that:
\begin{quote}
$\tl\in \{1/\lambda_-(q),1/\lambda_+(q)\}$ and hence $\tl=1/\lambda_-(q)$.
\end{quote}
This establishes that $\tl$ is indeed unique and completes the proof of Theorem~\ref{theorem}\eqref{theorem:a}. 

Next, note that $\lambda_-:[0,\infty)\to (-1,0)$ is certainly continuous (from the right) at $0$. Then in the relation $\Theta(1/\lambda_-(q))-q=0$, $q>0$, we may let $q\downarrow 0$ and see that  $\Theta(1/\lambda_-(0))=0$. By the above argument it is clear that $1/\lambda_-(0)$ is hence the unique zero of $\Theta$ on $(-\infty,-1)$. From Section~\ref{section:setting}, we also know that $1/\lambda_+(0)=e^{\Phi(0)}$ is the largest zero of $\Theta$ on $[1,\infty)$. Thus Theorem~\ref{theorem}\eqref{theorem:b} follows directly from \eqref{eq:3} and \eqref{eq:4}. 

%Theorem~\ref{theorem}\eqref{theorem:band1/2} is clear. 

Theorem~\ref{theorem}\eqref{theorem:c} is a corollary to Theorem~\ref{theorem}\eqref{theorem:b}.

Let us now establish Theorem~\ref{theorem}\eqref{theorem:d}. First, $X$ drifts to $+\infty$ or oscillates, if any only if $\PP(T_1<\infty)=1$ (equivalently, $\PP(T_n<\infty)=1$, for all natural $n$). An elementary computation based on Theorem~\ref{theorem}\eqref{theorem:b} then implies this is further equivalent to $\Phi(0)=0$, i.e. $\psi'(0+)\geq 0$. Assume now $\Phi(0)=0$. Note $X$ drifts to $+\infty$, if and only if $\EE[T_1]<+\infty$ (equivalently, $\EE[T_n]<+\infty$, for all natural $n$) \cite[p. 172, Proposition VI.17]{bertoin}. But in \eqref{eq:5}, which gives $\EE[e^{-q T_n}\mathbbm{1}(T_n<\infty)]=p_n(q)$, we may, via monotone convergence, differentiate under the integral sign with respect to $q$ (at $q=0$ from the right), to obtain: $$\EE[T_1]=\lim_{q\downarrow 0}\frac{1-\lambda_+(q)}{q}(1-\lambda_-(q))=-\lambda_+'(0+)(1-\lambda_-(0))=\Phi'(0+)(1-\lambda_-(0)),$$ which is finite if and only if $\psi'(0+)>0$. 

Finally, note that Corollary~\ref{remark:remark} follows from Theorem~\ref{theorem}, by first extending the latter so that $\lambda(\mathbb{R})$ is not necessarily equal to $1$ (simply apply Theorem~\ref{theorem} to the process $X_{ \cdot/\lambda(\mathbb{R})}$, whose L\'evy measure \emph{has} mass one), then applying Theorem~\ref{theorem} to the process $X$ with L\'evy measure $(1-p_0)\sum_{n\in\mathbb{Z}\backslash \{0\}}\frac{p_n}{1-p_0}\delta_n$ (so that for each $n\in \mathbb{N}\cup \{0\}$, $T_n$ becomes equal in law to the independent sum of $\tilde{T}_n$ iid, exponential, mean one, random variables). \qed

\section{Applications in queues and branching processes}\label{section:application}
Right-continuous random walks are related (at least in the distributional sense) to certain quantities in the theory of queues, and branching processes, see e.g. \cite[Section 5]{pitman} for a nice exposition. For the `nearly right-continuous' case, we offer the following two examples in applications.

Consider first a single queue of customers with \emph{two} equally capable servers, the latter attending to the former simultaneously, two at a time, per service. There are a total of $k\geq 2$ individuals in the queue at the start, $Q_0:=k$. The time the servers are working consists of idle and busy periods, where we define an idle period as a period in which at least one of the servers has no customer to attend to (and then only one server performs the service, say, while the other one rests). Thus, each busy period consists of one or more services; one service per two customers. Let $Q_n$ denote the number of customers in the queue at the end of the $n$-th service. Assume that the number of individuals which arrive during each service is distributed according to the distribution function $F$ ($\mathrm{d}F$ supported by $\mathbb{N}\cup \{0\}$) and that arrivals during each service period are independent (in their number). Let $(X_i)_{i\geq 1}$ be an independency of random variables distributed according to $F$ and $(S_n)_{n\geq 0}$ be their partial sums. Consider the process $P_n:=k+S_n-2n$ ($n\geq 0$), which is to model $(Q_n)_{n\geq 0}$ up to the idle period. If we let $T_k:=\inf \{n\geq 0: P_n\leq 1\}$, then (with equality in distribution) $T_k$ is the total number of services during the first busy period; accordingly it is also the time to the first idle period, and $2T_k$ is the number of customers served during the first busy period. 

Crucially, $(S_n-2n)_{n\geq 0}$ is nothing else than the negative of a `nearly right-continuous' random walk, and $T_k$ are precisely its first passage times. Note how e.g. Theorem~\ref{theorem}\eqref{theorem:c} implies that the servers will, at least for all sufficiently large $k$, always be more likely to end up in their first idle period with both being able to rest, rather than just one of them getting a chance to do so (assuming, for example, $\mathrm{d}F(\{0\})\mathrm{d}F(\{1\})>0$). 

As our second application, note that we can also find in the above queue an example of what is essentially (but not quite) a Galton-Watson branching process for \emph{paired} individuals, in the following precise sense. Consider having a totality of $k\geq 2$ initial ancestors in the $0$-th generation, which reproduce in pairs, each pair giving young to a certain number of descendants of the next generation, independently (in offspring number), and according to the distribution $F$. All the pairs are assumed disjoint. Note in each generation there is of course the possibility of having an individual, which cannot be paired up. How precisely to treat him will soon become clear, once the connection to the above has been established. 
%At any rate, at least so long as the number of individuals is much greater than $2$, to a first-order approximation, this (unfortunate) specimen can be neglected.

Now, we say the population becomes extinct if there are no longer two individuals present (in the current generation), which can pair up and reproduce. Then in the above (with equality in distribution, and as an approximation) $2 T_k$ represents the total progeny (modulo, possibly one member) until extinction has occurred (interpret customer $j$ a child of $j'j''$, if $j$ has arrived during the service of $j'j''$). The approximation is in that if the total number of individuals in a generation $K$ is odd, then the left-over specimen $j'$ in generation $K$ can reproduce with a member $j''$ of generation $K+1$ (provided, of course the left over pairs of generation $K$ have produced any progeny). In that case, if any progeny occurs between $j'$ and $j''$, it is assumed to be added to the generation $K+2$, which follows the oldest generation of this pair. Of course then $j''$ is no longer available for reproduction in generation $K+1$. This continues until there are still individuals available to reproduce. 

Indeed, if the branching process is defined in this latter sense (so a Galton-Watson branching for pairs, with suitable boundary conditions dealing with the possibility of having an odd number of individuals available (left over) for reproduction in the current generation), then the correspondence is exact and $2 T_k$ represents the total progeny (modulo, possibly, one member) until extinction has occurred. Theorem~\ref{theorem}\eqref{theorem:d}, for example, then gives information on whether the total progeny will be a.s. finite or not/of finite expectation or not. 

\section{Conclusion}\label{section:conclusion}
It would be interesting to see what (if anything definitive) can be said, when the jumps of $X$ are allowed upwards up to a certain (fixed, but arbitrary) threshold $N\in\mathbb{N}$ (we had $N=2$, $N=1$ being the skip-free case). This remains open to future research. 
\bibliographystyle{plain}
\bibliography{Biblio_first_passage_nearly_rightcts_RW}

\end{document}